\documentclass{article}

\usepackage{amssymb,amsmath,amsthm}
\usepackage{hyperref}
\usepackage{enumitem}
\usepackage{graphicx, comment}
\usepackage[normalem]{ulem}
\usepackage{url}

\usepackage{tikz}
\newcommand*\circled[1]{\tikz[baseline=(char.base)]{
    \node[shape=circle,draw,inner sep=2pt] (char) {#1};}}

\newcommand{\N}{\mathbb{N}}

\usepackage{thmtools, thm-restate}

\declaretheorem[name=Theorem]{thm}

\newtheorem{lemma}[thm]{Lemma}
\newtheorem{corollary}[thm]{Corollary}

\newtheorem{theorem}[thm]{Theorem}

\newtheorem{observation}[thm]{Observation}

\newtheorem{problem}{Problem}

\theoremstyle{definition}
\newtheorem{example}[thm]{Example}

\newcommand{\remove}[1]{}
\renewcommand{\epsilon}{\varepsilon}

\title{Counting Houses of Pareto Optimal Matchings\\  
in the House Allocation Problem}

\author{Andrei Asinowski \and Bal\'{a}zs Keszegh \and Tillmann Miltzow}

\date{\vspace{-5ex}}

\begin{document}

\maketitle

\begin{abstract}
In an instance of the house allocation problem,
 two sets $A$ and $B$ are given. The set $A$ is referred to as \emph{applicants} and the set $B$ is referred to as \emph{houses}. We denote by $m$ and $n$ the size of $A$ and $B$ respectively.  In the house allocation problem, we assume that every
 \emph{applicant} $a\in A$ has a preference list over the set of houses $B$.
We call an injective mapping $\tau$ from $A$ to $B$ a matching. A \emph{blocking
 coalition} of $\tau$ is a non-empty subset $A'$ of $A$ such that there exists a matching
  $\tau'$ that differs from $\tau$ only on elements of $A'$, and every element of $A'$ improves in $\tau'$, compared to $\tau$, according to its preference
 list. If there exists no blocking coalition, we call the matching $\tau$ a
  \emph{Pareto optimal matching} (POM).

A house $b\in B$ is {\em reachable} if there exists a Pareto optimal matching using
 $b$. The set of all reachable houses is denoted by $E^*$. We show
  \[|E^*| \leq  \sum_{i = 1,\ldots, m}{\left\lfloor\frac{m}{i}\right\rfloor} = \Theta(m\log m).\]
 This is asymptotically tight. A set $E\subseteq B$ is \emph{reachable} (respectively
 \emph{exactly reachable})  if there exists a Pareto optimal matching $\tau$ whose
  image contains  $E$ as a subset (respectively equals $E$). We give bounds for the
   number of exactly reachable sets.
We find that our results hold in the more general setting of multi-matchings, when each applicant $a$ of $A$ is matched with $\ell_a$ elements of $B$ instead of just one. Further, we give complexity results and algorithms for corresponding algorithmic questions. Finally, we characterize {\em unavoidable} houses, i.e., houses that are used by all POMs.
We obtain efficient algorithms to determine all unavoidable elements.
\end{abstract}

\section{Introduction}

\subsection{Definitions}

The house allocation problem is motivated by the following setup: a set of people wish to be allocated to a certain set of houses. Each person ranks the set of houses and wants to be assigned to the house with her highest preference. As soon as two people have the same favorite house this is not possible. Motivated by this picture we generalize the 
setup. 
We start with some definitions.

In an instance of the house allocation problem, 
two sets $A$ and $B$ are given. The set $A$ represents applicants and the set $B$ represents houses.
We denote by $m$ and $n$ the size of $A$ and $B$ respectively.  In the house allocation problem, we assume that every $a\in A$ has a preference list over the set $B$. A preference list can be formally defined as a total order 
on $B$.
We call an injective mapping $\tau$ from $A$ to $B$ a {\it matching}. A \emph{blocking
 coalition} of $\tau$ is a non-empty subset $A'$ of $A$ such that there exists a matching
  $\tau'$ that differs from $\tau$ only on elements of $A'$, and every element of $A'$ improves in $\tau'$, compared to $\tau$ according to its preference
 list. If there exists no blocking coalition, $\tau$ is called a
  \emph{Pareto optimal matching} (POM).


%
%


We represent the preference lists by an $m\times n$ matrix. Every row represents the preference list of one of the applicants in $A$, i.e., in a given row $r$ corresponding to some applicant $a\in A$, the leftmost house is the one that $a$ prefers most, etc.,
that is house $b_1$ is left of $b_2$ in $r$ if and only if $a$ prefers $b_1$ over $b_2$.
Note that no row contains an element from $B$ twice. We usually denote this matrix by $M$ and following this interpretation we usually denote the members of $A$ by $r_1,r_2,\dots r_m$ and the members of $B$ by $1,2,\dots,n$. Because of this matrix representation, we usually refer to members of $A$ only as rows and to members of $B$ as elements (of the matrix).

To illustrate the notion consider the following matrix and observe that the matching indicated by circles is indeed Pareto optimal.
\[
\left(
\begin{array}{ccccc}
        \circled{1} & 5 & 3 & 2 & 4 \\
    3 & 1 & \circled{4} & 5 & 2 \\
    1 & \circled{3} & 5 & 4 & 2 \\
\end{array}
\right)
\]
We denote tuples $p=(a,k)$ as positions of the matrix, if $a$ is some row and $k$ is some natural number. A matching corresponds to a set of positions $P_\tau$ if there is exactly one position for each row and no two positions contain the same house.
The image set of $\tau$ corresponds to the set of houses of $B$ in these positions. Thus, we say that $\tau$ \emph{selects}
some position $p$ of $M$ (resp.\ some element $b$ of $B$),  if $p$ is in $P_\tau$
(resp.\ $b$ is in the image set of $\tau$).
Similarly, given some matching $\tau$ we say that a row $a$  \emph{selects}  a position $p$ in row $a$ (resp.\ element $b$) if $p=(a,k)\in P_\tau$ for some $k$ (resp.\ $b$ is in the image of $\tau$).
We denote by $s(\tau)$ the image set of $\tau$.

In a POM
the positions after the $m$-th column will never be assigned, because at least one of the previous $m$ elements in that row is preferred and not assigned to any other element on $A$. Therefore it is sufficient to consider only $m\times m$ square matrices.
(In other words, only the first $m$ elements of the preference lists matter.)

If some POM $\tau$ selects $p$ (resp.\ $b$), then it is a {\em reachable} position (resp.\ {\em reachable} element). More generally, a set $E \subseteq B$ is \emph{(exactly) reachable} if there exists a POM $\tau$ with $E\subseteq s(\tau)$ ($E = s(\tau)$). In this case we also say that $\tau$ reaches $E$.
An element $b$ is \emph{unavoidable} if it belongs to the set $s(\tau)$ for every  Pareto optimal matching $\tau$ of $M$ and $b$ is called {\em avoidable} if there exists a Pareto optimal matching $\tau$ with $b\notin s(\tau)$. A set $E$ is {\em avoidable} if there exists a POM $\tau$ with $s(\tau)\cap E = \varnothing$.
Note that for a set $E$ such that $|E|=m$ it is exactly reachable if and only if $B\setminus E$ is avoidable.
These notions can be generalized naturally to the case of {\em multi-matchings}, when each element $a$ has to be paired up with $\ell_a\ge 1$ elements of $B$. For the precise definitions see Section~\ref{sec:esmm}.
 These {\em Pareto Optimal Multi-matchings (POMMs)} can also appear naturally in practical applications and we will see that our results about POMs generalize naturally to POMMs.
We will also study matrices with fewer than $m$ columns, precise definitions will be given in Subsection~\ref{sibsec:prelim}. In this case preference lists are incomplete, i.e., it can happen that some elements of $A$ are not assigned.
(In this case, it might be interesting to compute a POM of maximum size.)

\begin{example}
To illustrate the notions and to avoid confusion,
we give here a detailed example. The reader can use it
to verify that she understood all important notions.
The example can be skipped, as we will not refer to it again.
Consider the following matrix.
\[
M = \left(
\begin{array}{cccc}
        1 & 5 & 3 & 2  \\
	3 & 1 & 5 & 4  \\
	1 & 6 & 5 & 4  \\
	3 & 6 & 2 & 4  \\
\end{array}
\right)
\]

The elements $1$ and $3$ are unavoidable, as they are both
in the first column of $M$ and thus picked by every POM.
A quick check reveals that every other element of $M$ can
be avoided.
Each set $E$ with $1\in E$ or $3\in E$ is also
unavoidable. A simple argument reveals that
every set $E \subseteq \{2,4,5,6\}$ with $|E| \geq 3$ is
unavoidable, because every POM picks $4$ elements of $M$.
In order to determine all unavoidable sets,
there remain only six sets $E \subset \{2,4,5,6\}$  of size $2$ to
be examined. It turns out that only  $\{5,6\}$ and $\{2,5\}$
are unavoidable.

It is easy to see that all elements of $M$ are reachable.
In order to specify the reachable sets, note that if some
set $E$ is reachable then are all its subsets as well.
Thus it is sufficient to specify all reachable sets of size $4$.
Also note that if $E \subseteq \{1,2,3,4,5,6\}$ of size four
is reachable then $D = \{1,2,3,4,5,6\} \setminus E$
is avoidable. Thus, by the discussion above, the
four reachable sets of size four are exactly
\[ \{1,3,4,5\}, \{1,3,2,5\}, \{1,3,2,6\}, \{1,3,5,6\}. \]

\end{example}

\subsection{Results}

\subsubsection{Enumerating reachable elements and sets}
In Section \ref{reachable} we deal with enumerative problems related to reachable elements. Our main result here is the following.
\begin{restatable}{thm}{thmElements}\label{thm:elements}
Let $M$ be an $m\times m$ matrix and $E^*$ be the set of all reachable elements. Then \[\left| E^* \right| \leq \sum_{i=1}^{m}{\left\lfloor m/i\right\rfloor} \leq m (\ln m +1) .\]
\end{restatable}

This improves the trivial upper bound of $m^2$ 
which appears 
in~\cite{henzecomplexity}. 
In~\cite{henzecomplexity}
the authors also presented a lower bound construction, which has asymptotically as many reachable elements as is implied by our upper bound.
Thus Theorem~\ref{thm:elements} is asymptotically tight.
	
Denote by $\mathcal{E}(M)$ the family of all (exactly) reachable $m$-element sets of $M$.
For example, if all the elements in the first column of $M$ are distinct,
then $|\mathcal{E}(M)| = 1$.
With Theorem~\ref{thm:elements} we can bound $\mathcal{E}(M)$.

\begin{corollary}\label{crl:mlogm_choose_m}
For any matrix $M$, we have $|\mathcal{E}(M)| \le \binom{m (\ln m +1)}{m}$.
\end{corollary}

This is the only non-trivial upper bound that we found, improving $\binom{m^2}{m}$ of \cite{henzecomplexity}. As an important consequence, our upper bound also improves the upper bound on the pattern matching problem regarded in \cite{henzecomplexity}. The best known lower bound is $\binom{m}{\lceil m/2\rceil}$ \cite{henzecomplexity}.
The construction in that paper is a matrix where in the first $\lfloor m/2\rfloor$ columns the $i$-th column contains only element $i$ and in the $(\lfloor m/2\rfloor+1)$-st column there are $m$ different elements which are also all different from $1,2,\dots , \lfloor m/2\rfloor$.


\subsubsection{Characterization of avoidable elements and sets}
 Section~\ref{sec:avoid} concentrates on the notion of avoidable elements.
Let $x$ be an element suspected to be avoidable.
Given some set of rows $R$ we denote by $E_x(R)$ the set of elements left of $x$ in the rows $R$ (i.e., $y$ is in $E_x(R)$ if and only if there exists a row $r\in R$ in which $y$ appears to the left of $x$; if $x$ does not appear in $r$ then all elements in $r$ are regarded to be left of $x$).

\begin{restatable}{thm}{thmUnavoidable}\label{thm:unavoidable}
An element $x$ of a matrix $M$ is avoidable if and only if
for every set $R$ of rows of $M$, we have:
	\[|E_x(R)| \geq  |R| .\]
\end{restatable}

Extremal and algorithmic results
in connection to avoidable elements are included in Section~\ref{sec:avoid}, as they follow from the proof of Theorem~\ref{thm:unavoidable}. We prove the following corollary.
%
%
%
\begin{restatable}{corollary}{corAlgorithmUnavoidable}\label{thm:unavoidablecompl}
Deciding whether an element $x$ is avoidable can be done in\\  $O(m^2\sqrt{m+n})$ time and also in $O(m^3)$ time. Listing all unavoidable elements can be done in  $O(m^2n\sqrt{m+n})$ time and also in $O(m^5)$ time.
\end{restatable}

Both results follow from an easy reduction to matchings in a bipartite graph.

\subsubsection{Complexity of reachability}
Computational questions about reachable elements are considered in Section~\ref{sec:Compl}.
Among the computational questions connected to the notions
defined in this paper, we selected those
that we found natural and canonical.
%
The problems are defined as follows:


\begin{problem} \label{prb:DR}(Deciding Reachability)\newline
\textbf{Input:} A matrix $M$ and a set $D\subseteq B$.\\ \noindent
\textbf{Question:} Is $D$ reachable?
\end{problem}


\begin{problem} \label{prb:DER}(Deciding Exact Reachability)\newline
\textbf{Input:} A matrix $M$ and a set $E\subseteq B$\\ \noindent
\textbf{Question:} Is $E$ exactly reachable?
\end{problem}


\begin{problem}\label{prb:CR} (Counting Reachable Sets)\newline
\textbf{Input:} A matrix $M$.\\ \noindent
\textbf{Question:} How many sets $D\subseteq B$ are reachable?
\end{problem}


\begin{problem} \label{prb:CER}(Counting Exactly Reachable Sets)\newline
\textbf{Input:} A matrix $M$.\\ \noindent
\textbf{Question:} How many sets $E$ are exactly reachable?
\end{problem}


\begin{problem} \label{prb:CERSS}(Counting Exactly Reachable Supersets)\newline
\textbf{Input:} A matrix $M$ and some set $D\subseteq B$.\\ \noindent
\textbf{Question:} How many sets $E$ with $D\subseteq E\subseteq B$ are exactly reachable?
\end{problem}


\noindent The next table summarizes our
findings regarding algorithmic questions.
The general case is always the same as with $3$ column matrices.
The hardness results for Problems~\ref{prb:DR} and~\ref{prb:CERSS}
already hold if $D$ contains exactly $1$ element.
It was already observed by Henze, Jaume and Keszegh that Problem 1 is NP-complete~\cite{henzecomplexity}. Our contribution, among others, is to show NP-completeness also for matrices with only $3$ columns. The running time of Problem~\ref{prb:DER} follows from the discussion in Section~\ref{sec:avoid} and Corollary~\ref{thm:unavoidablecompl}.
\begin{center}
\begin{tabular}{|c|cc|cc|}
\hline
Problem & 2 columns & proof & 3 columns & proof \\
\hline
\ref{prb:DR})  & polynomial & (Thm.~\ref{thm:polyreachable}) & NP-complete & (Thm.~\ref{thm:ReductionSAT}) \\
\ref{prb:DER})  & polynomial &(Cor.~\ref{thm:unavoidablecompl}) & polynomial &(Cor.~\ref{thm:unavoidablecompl}) \\
\ref{prb:CR}) &  explicit formula &(Thm.~\ref{lem:numberSelectable}) & ? & \\
\ref{prb:CER})  & \#P-complete & (Thm.~\ref{thm:problem5}) & \#P-complete & (Thm.~\ref{thm:problem5}) \\
\ref{prb:CERSS})   &  \#P-complete & (Thm.~\ref{thm:problem5}) & \#P-complete &(Thm.~\ref{thm:ReductionSAT}) \\
\hline
\end{tabular}
\end{center}

It remains an open question whether Problem~\ref{prb:CR} is hard for general matrices. We conjecture it is already \#P-complete for $3$ column matrices.

Problem~\ref{prb:CER} is a special case of Problem~\ref{prb:CERSS} for the case 
 that $D = \varnothing $.

\subsection{Motivation and related work}

Pareto optimal matchings are one of many concepts to
define a  `good' matching in a graph with preference lists.
The graph properties and the preference lists depend on the
precise situation they are supposed to model.
Most research in this area is driven by economic applications.
%
%
A recent book on matchings under preferences, by David Manlove~\cite{manlove2013algorithmics}, contains a broad overview of the field and usually focuses on algorithmic rather than economic questions.
In this paper we tried, whenever applicable, to follow the notation therein.


A field that evidently seems to be related to our topic is that of stable matchings. This field is very broad and belongs to economic game theory. The seminal work of Gale and Shapley is the starting point for this field~\cite{gale1962college}. Some work in this field and different variations of the problem can be found in the Ph.D. thesis of Sandy Scott~\cite{scott2005study}; recent papers can be found in the proceedings of the Second International Workshop on Matching Under Preferences called MATCH UP~\cite{MATCHUP}. 
In these works there are many different concepts of preferences and stability and they ask for efficiently computable solutions that maximize the outcome for the participants in one way or the other. Readers interested more broadly in the topic of algorithmic game theory are referred to the book edited by Nisan, Roughgarden, Tardos and Vazirani~\cite{nisan2007algorithmic}.

In contrast to most research done in these areas, our question is more
combinatorial in nature. The underlying algorithmic question of computing
a Pareto optimal matching can be solved by the greedy algorithm. Thus,
instead of existence questions, rather the enumerative questions become
interesting.
(Note that it is non-trivial to compute POMs with extra features.)
Nevertheless, for the early definition of stability many authors have
tried to establish upper and lower bounds on the number of stable matchings and some combinatorial structures have been revealed. See \cite[Section 2.2.2]{manlove2013algorithmics} for an overview of results in this direction.

Some complexity results similar to ours have been found earlier. The first dates back to 2005~\cite{abraham2005pareto}. The 
main result in this work is an efficient algorithm to compute a POM with maximum cardinality~\cite{abraham2005pareto}. Here the preference lists are incomplete.
(See Section~\ref{sibsec:prelim} for precise
definitions of incomplete preference lists.)
Additionally, the authors of~\cite{abraham2005pareto}
show a hardness result for computing a minimum POM. This result has already some ideas of the proof of Theorem~\ref{thm:unavoidable}. Further, the proof of Theorem~\ref{thm:problem5} implies hardness of computation of minimum  matchings. Although they show an easy $2$-approximation, it is open whether there exists a PTAS for a minimum POM.

We are aware of four recent papers that considered similar results to our complexity
results~\cite{saban2015complexity, aziz2013computational, aziz2014parametrized, cechlarova2013pareto}.
Their main motivation is to study the behavior of the randomized serial dictatorship also called randomized priority allocation. The randomized serial dictatorship picks a permutation at random and thereafter computes the corresponding greedy matching.

The first paper is by Sab\'{a}n and Sethuraman~\cite{saban2015complexity}. Their result, reformulated in our context, is NP-hardness of Problem~\ref{prb:DR}, for arbitrary matrices. Aziz, Brandt and Brill~\cite{aziz2013computational} show \#P hardness for a variant of Problem~\ref{prb:CERSS} for arbitrary matrices. We improve these results, as we can show this holds also for matrices with only $3$ columns. Aziz and Mestre show that constraint versions are solvable in polynomial time~\cite{aziz2014parametrized}. Cechl\'arov\'a et al.~\cite{cechlarova2013pareto} consider a generalized setting:
they show NP-hardness of computing a minimum maximal matching even for matrices with $2$ columns by an elegant reduction from vertex-cover very similar in spirit to the proof in Theorem~\ref{thm:problem5}.

This work is originally motivated by work that was presented at the EuroCG 2012 in Braunschweig~\cite{henzecomplexity}. See also the follow up work~\cite{henzejournal}. The authors
considered a generalisation of Voronoi diagrams under the assumption that not just one point, but many points are matched in a way that minimizes
the sum of the squares of distances between matched points. From the definitions in their paper, the Pareto optimality comes as a natural property. They asked explicitly for the number of exactly reachable sets, as it gives an upper bound on the number of Voronoi cells in the above setting. Motivated by this, they gave lower and upper bounds on the number of exactly reachable stable sets. To do this, first they gave lower and upper bounds
for the number of reachable elements. In this paper we improve their upper bound for the number of reachable elements and we prove that their lower bound is asymptotically correct. This also yields a significant improvement on the previous upper bound on the number of exactly reachable stable sets, although in this case our new upper bound still does not meet the lower bound they had.
Their work is based on work by Rote presented at the EuroCG 2010  in Dortmund~\cite{rote2010partial}.

\subsection{Preliminaries}\label{sibsec:prelim}

As we also want to study matrices with fewer than $m$ columns, we need to define what we mean by a matching under these assumptions. There
 are two equivalent ways. First, we could say that every row, for which all elements are already picked by other rows, just does not get assigned to anything.
A second and nicer way is to add columns, with all elements in one column being the same and not appearing before. If we want to know whether some set $E$ is exactly reachable in the first way, we construct $E'$ from $E$ by adding the elements from the first $m- |E|$ additional columns (and vice versa).
The following is an example of a $2$ column matrix.
\[
\left(\begin{array}{cccc}
1 & 4 \\
2 & 1 \\
2 &5  \\
4 & 3 \end{array}\right)
\sim
\left(\begin{array}{cccc}
1 & 4& c_1 & c_2 \\
2 & 1&c_1 & c_2 \\
2 &5 & c_1 & c_2 \\
4 & 3& c_1 & c_2 \end{array}\right)
\]

We use the first approach. However, using the second approach, some hardness results will carry over from $2$ or $3$ column matrices to $k$ column matrices ($2\leq k \leq m$). In such a case, we will point this out again at the appropriate places.

To see the correspondence between matchings in a graph-theoretical sense and in our context we define the \emph{bipartite row element graph} $G$ as follows. The vertices are defined as the set of rows and elements; an element $e$ is adjacent to some row $r$ if and only if $e$ appears in $r$. See the example for the special case of a matrix with only $2$ columns.
\begin{center}
	\includegraphics{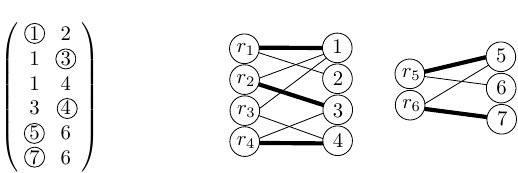}
\end{center}
The circled POM corresponds to the thick matching on the right side.

If there is no blocking coalition of size 
at most $i$, 
we call the matching an \emph{$i$-Pareto optimal matching} ($i$-POM).
In particular, this implies that every POM is an $i$-POM.
We call a matching \emph{$1$-POM} if there is no blocking coalition of size one.
The next matching is $1$-Pareto optimal but not Pareto optimal.
\[
\left(
\begin{array}{ccc}
        1 & \circled{5} & 3  \\
    5 & 1 & \circled{4}  \\
    5 & \circled{1} & 2 \\
\end{array}
\right)
\]

\begin{observation}\label{obs:triv}
If $\tau$ is a POM and $\tau$ selects position $p$ in row $a$, then $\tau$ selects every element that appears in row $a$ left of $p$.
\end{observation}
This observation also holds for $1$-POMs.
A matching $\tau$ is \emph{greedy} if there exists a permutation $\pi$ of $A$ such that the matching can be generated in the following manner: we process the rows of $M$ in the order determined by $\pi$, and in each row we pick the leftmost element that was not picked earlier.
Given some permutation $\pi$ we call the corresponding greedy matching $\tau_\pi$.

In the literature, the mechanism of determining a matching from a permutation
is called the \emph{Serial Dictatorship Mechanism}.

Lemma~\ref{lem:equivNotion} brings all the introduced notions together, showing that POMs, 1-POMs and greedy matchings select exactly the same sets.

Actually, POM and greedy matchings are equivalent notions, as proven in~\cite{abraham2005pareto, henzecomplexity, Svenson1994}, and probably also by many others.

For the proof  of the next lemma we need an equivalent definition of POM.
Assume $\tau_1 \neq \tau_2$.
We say $\tau_1$ is \emph{better than} $\tau_2$ if every element $a\in A$ is matched in $\tau_1$ to a better or equally good element than in $\tau_2$, according to the preference list of $a$.  POMs are exactly the maximal elements with respect to this \emph{better than} relation.

\begin{lemma}\label{lem:equivNotion}
  Let $E\subseteq \{1,\ldots,n\}$ with $|E|=m$. The following statements are equivalent.
  \begin{enumerate}
   \item $E$ is (exactly) reachable, i.e.\ there exists a POM $\tau$ with $s(\tau) = E$.
   \item There exists a permutation $\pi$ such that for the greedy matching $\tau_\pi$ we have $s(\tau_\pi) = E$.
   \item There exists a $1$-Pareto optimal matching (1-POM) $\tau$ with $s(\tau) = E$.
  \end{enumerate}
\end{lemma}
\begin{proof}
  $[1 \Rightarrow 2]$
Let $\tau$ be a POM matching such that $s(\tau)=E$.
  We construct a permutation $\pi$ inductively. If possible take as the next row, in the order of our permutation, the one that has a position of $\tau$ in its first entry. Delete the element $a$ at this position from all other rows and continue. We show that at each stage there must be such a row.
  For the purpose of contradiction assume such a row does not exist. Take any row, denoted by $q_1$,
  and let $e_1$ be some element which is left to the element selected by $\tau$ in row $q_1$.
  Because $\tau$ is Pareto optimal, there exists some row $r_2$ selecting $e_1$. Let $e_2$ be any element left to $e_1$ in row $r_2$. In this way we can define a sequence $(e_i)$ and $(r_i)$. As we have only finitely many elements, at some point we get a first $e_j$ that appears earlier in the sequence $e_i=e_j$, $i<j$. This implies that in the rows $r_i,\dots r_j$ we can improve simultaneously (i.e., it is a blocking coalition), which is a contradiction to the assumption that $\tau$ is Pareto optimal.

  $[2 \Rightarrow 3]$
  As every row picks the best element, not yet selected, it is clear that no single row can improve.

  $[3 \Rightarrow 1]$
  Let $\tau_0$ be some $1$-Pareto optimal matching and $E = s(\tau_0)$. Observe that all the elements left to the elements picked by $\tau_0$ are in $E$. The set of matchings that are better or equal to $\tau_0$ is non-empty as it contains $\tau_0$ and the set is of course finite, so there exists a best matching $\tau_1$ among them, i.e. one for which there is no better matching. This must be a POM and by Observation~\ref{obs:triv} we have $s(\tau_1)\subseteq s(\tau_0)=E$,
 and the size of $s(\tau_1)$ is also $m$. This implies $s(\tau_1) = E$.
\end{proof}
 Note that this lemma implies that also for any $i$, $i$-POMs select the same sets as POMs/$1$-POMs. Note also that the proof of Lemma~\ref{lem:equivNotion} implies that actually every POM matching is greedy. The inverse is also true and left as an exercise, as we will not use it later.







\section{Counting reachable elements and sets}\label{reachable}

We start with a discussion of reachable positions.
For every row $r$, there exists a reachable position $p_r$ furthest to the right in that row, we call such a position \emph{last reachable}. However  not all positions left of the last reachable position must be reachable. Consider for example the following matrix together with the matching $\tau$ indicated by circles.
\[
\left(\begin{array}{cccc}
\circled{5} & 4& 3 & 2 \\
5 & \circled{1} & 6 & 7 \\
1 &\circled{2} & 8 & 9 \\
2 & 1& 5 & \circled{4} \end{array}\right)
\]
The matching $\tau$ is  Pareto optimal and thus the circled position in the bottom row with element $4$ is the last reachable position in that row. However, it is easy to verify that the two positions left of this circled position (with elements $1$ and $5$) are not reachable.

\thmElements*


\begin{proof}

Let $T$ be some set of $k$ POMs of $M$.
We denote by $E(T)$ the set of elements reached by at least one POM of $T$.
(formally: $E(T) = \bigcup_{\tau \in T}s(\tau)$.)
We claim: \[\left| E(T) \right| \leq \sum_{i=1}^{k}{\left\lfloor m/i\right\rfloor}.\]
Note
\[ \sum_{i=1}^{k}{\left\lfloor m/i\right\rfloor} \leq \sum_{i=1}^{m}{\left\lfloor m/i\right\rfloor} \leq m (\ln m +1).
\]
 Because $\lfloor m/i \rfloor = 0$ for $i>m$ and the known bound on the harmonic series.
Thus the theorem follows from the claim.

The proof goes by induction on $k$.	The base case $k= 1$ is true as one POM selects exactly $m$ different elements.

Consider now  a set $T$ of $k\geq 2$ POMs and the set of positions reached by $T$.
Among these positions we denote by $p_i$ the position furthest to the right in row $i$ and we denote $F = \{p_1,\ldots ,p_m \}$.
We say that an element $e$ (resp.\ position $p$) is uniquely reachable by some $\tau$ if $\tau$ is the only POM in $T$ that reaches $e$ (resp.\ selects $p$). Consider the set $G\subseteq F$ of those rightmost reachable positions that are reachable by exactly one POM of $T$. By the pigeon-hole principle there exists a POM $\tau$ in $T$ that reaches at most $1/k$ portion of $G$.
 Denote the set of elements in these positions by $H$ ($|H|\leq \lfloor m/k \rfloor$).


By the definition of $H$ all elements $s(\tau) \setminus H$ are not selected uniquely by $\tau$, i.e. some other matching of $T$ also selects it.

Let us explain this in more detail. Consider an element $e \in s(\tau) \setminus H$ and denote by $p$ the position of $e$. As $p \not \in G$ there exists another matching $\tau' \neq \tau$ that either also selects $p$ or $\tau'$ selects a position further to the right of $p$ in the same row. In both cases holds: $\tau'$ must select $e$ by Observation~\ref{obs:triv}.

Thus the rest of the reached elements are also reachable by $T \setminus \tau$.
By induction we get
 \[E(T) \leq E(T \setminus \tau) \ + \lfloor m/k\rfloor\leq \left( \sum_{i=1}^{k-1}{\left\lfloor m/i\right\rfloor}\right) \ + \ \left\lfloor m/k\right\rfloor = \sum_{i=1}^{k}{\left\lfloor m/i\right\rfloor}. \qedhere \]
%
%
\end{proof}

A natural question is if we could asymptotically improve
upon the bound given in the previous Theorem~\ref{thm:elements}.
The following construction
by Henze, Jaume and Keszegh~\cite{henzecomplexity}
shows asymptotic tightness.

The following example has more rows than columns, however
the rows can be filled with arbitrary elements as
none of these elements are reachable anyway.

\begin{example}[\cite{henzecomplexity}]\label{ex:mlogmExample}
For each $k$, a matrix $M_k$
with $m=2^{k}$ rows and \\ $(m/2)\log 4m$ $=(k+2)2^{k-1}$ reachable elements is constructed recursively as follows.

   \[M_0 = \left(\begin{array}{c}
1
 \end{array}\right);\]
 and, for $k \geq 0$,
   \[M_{k+1} = \left(\begin{array}{c|c}
1 &  \\
\vdots & M_k'  \\
2^k & \\
\hline
1 &  \\
\vdots & M_k''  \\
2^k & \\
\end{array}\right),\]
where $M_k'$ and  $M_k''$ are relabelings\footnote{
A matrix $M'$ is a {\em relabeling} of a matrix $M$
if there is a bijective function between the elements (not positions!)
of $M$ and $M'$ such that applying this function to the elements in all the positions of $M$ we get $M'$.
Clearly two matrices that are relabelings of each other are equivalent from our perspective.
}
of $M_k$ with no common element and all elements different from $1,2,\dots, 2^k$.
The undefined entries of the matrix can be chosen arbitrarily.
\end{example}


Regarding the claim in the proof of Theorem~\ref{thm:elements}, we prove that it is tight for certain values of $k$ and $m$:
\begin{corollary}\label{cl:mlgmtight}
For every $k$ there exists a matrix $N_k$ with $m=k!$ rows
and a set $T_k$ of $k$ POMs,
such that the number of elements reached by $T_k$ is exactly $\sum_{i=1}^{k}{m/i}$.
\end{corollary}
\begin{proof}
The construction is again recursive.
For each $k$ we define the matrix $N_k$ with $m=k!$ rows and $k$ columns with the property that
each element appears only in one column,
and each element that appears in the $j$th column ($j\le k$), appears there exactly $k-j+1$ times.
We also define a set $\Pi_k$ of $k$ permutations of the $k!$ rows
from which we get $T_k$ by taking the greedy matchings corresponding to the permutations.
We will prove that all the elements of $N_k$ are reachable by some greedy matching of $T_k$.

The matrices $N_k$ are defined in the following way:
\[N_1 = \left(\begin{array}{c}1
\end{array}\right),\]
and for $k \geq 1$:
 \[N_{k+1} = \left(\begin{array}{c|c}
1 &  \\
\vdots & N_k^1  \\
k! & \\
\hline
\vdots & \vdots \\
\hline
1 &  \\
\vdots & N_k^{k+1}  \\
k! & \\
\end{array}\right).\]
Here $N_k^1,N_k^2,\dots ,N_k^{k+1}$ are $k+1$ matrices which are all relabelings of $N_k$
without common elements and no matrix $N_k^i$ contains an element from the set $\{1,2,\dots, k!\}$.
It is clear that $N_{k+1}$ has $(k+1)!$ rows and $k+1$ columns.
Moreover, each element in the $j$th column ($j\le k+1$) appears there $(k+1)-j+1$ times:
this is clear for the first column,
and is easily seen for other columns by induction.

Next we define the permutations.
For $k=1$, $\Pi_1$ contains the only permutation on the one row of $N_1$.
Next we recursively define $\Pi_{k+1}$.
For each $N_k^j$ ($1\le j\le k+1$), we have by recursion an associated set
$\{\pi^{j}_{1},\pi^{j}_{2}\dots \pi^{j}_{k}\}$ of $k$ permutations
(thus, $\pi^{j}_{i}$ is the $i$th permutation from $\Pi_{k}$ relabeled accordingly to $N^j_k$ --
the $j$th copy of $N_k$ in $N_{k+1}$).
Now the permutations in $\Pi_{k+1}$ are defined as follows.
For every $i$ ($1\le i\le k+1$),
the permutation $\pi_i$ is obtained by taking
first the $k!$ rows of $N_k^i$ in any (for example, the natural) order;
then the rows
of $N_k^1 \cup N_k^2 \cup \dots \cup N_k^{i-1}$
in the order determined by the permutations $\pi^{1}_{i-1},\pi^{2}_{i-1}, \dots ,\pi^{i-1}_{i-1}$;
and, finally, the rows
of $N_k^{i+1} \cup N_k^{i+2} \cup \dots \cup N_k^{k+1}$
in the order determined by the permutations
$\pi^{i+1}_{i}, \pi^{i+2}_{i}, \dots, \pi^{k+1}_{i}$.
Clearly, each row was taken once, so $\pi_i$ is indeed a permutation. Also, when processing such a permutation, in the first $k!$ steps we choose all elements $1,2,\dots, k!$, so in the rest the permutation chooses the same elements in each $N_k^j$ ($j\ne i$) as the corresponding permutation ($\pi^{j}_{i-1}$ or $\pi^{j}_{i}$) would choose in $N_k^j$.

Thus by induction it is true that these permutations choose all elements of $N_{k}$.
 Indeed, this is true for $N_1$ and by induction it remains true as for every $N_k^j$
 ($1 \le j\le k+1$) all $\pi^{j}_{i}$ ($1\le i\le k$) is part of some $\pi_u$ ($1\le u\le k+1$).
 Finally, the number of different elements in $N_{k+1}$ is
 $\sum_{i=1}^{k+1}{m/i}$, as we have $k+1$ columns, and in the $j$th column
($1 \le j\le k+1$) each element appears $(k+1)-j+1$ times, thus this column has $\frac{m}{(k+1)-j+1}$ different elements.
\end{proof}

\section{Characterization of avoidable elements}\label{sec:avoid}

In this section we characterize avoidable elements and sets.
Recall that we define $E_x(R)$ as the set of elements left of $x$ in the rows of $R$ (i.e., $y$ is in $E_x(R)$ if and only if there exists a row $r\in R$ in which $y$ appears to the left of $x$; if $x$ does not appear in $r$ then all elements in $r$ are regarded to be left of $x$)

\thmUnavoidable*

\begin{proof}
$[\Rightarrow]$ Let $\tau$ be a POM which does not pick $x$ and
let $R$ be a set of rows. In each row a different element is picked by $\tau$, which is left of $x$.
This shows the claim.

$[\Leftarrow]$
Without loss of generality $x$ is present in all the rows.
Consider the bipartite graph on $A\cup B$, defined by all pairs $(a,b)\in A\times B$ such that $b$ appears in row $a$ before $x$. The above condition says, that for all subsets $R\subset A$ the neighbourhood of $R$ is larger or equal to $R$ in terms of size.

By Hall's theorem, there exists a matching $\tau$ that picks elements to the left of $x$.
We can modify $\tau$ so that each row picks an element farthest to the left in $M$ not chosen by any other row. In other words $\tau$ is a $1$-POM.
By Lemma~\ref{lem:equivNotion} there is a POM $\tau'$ selecting the same set of elements as $\tau$, thus $\tau'$ does not choose $x$, and so $x$ is avoidable.
\end{proof}

%

\corAlgorithmUnavoidable*



\begin{proof}
For an element $x$ to be avoidable is equivalent to the existence of a matching
that connects all the elements of $A$ with elements left to $x$.
Thus, we need to find a maximum bipartite matching.
The fastest known algorithm
(Hopcroft-Karp~\cite{DBLP:journals/siamcomp/HopcroftK73})
checks this in $O(\sqrt{V} E)$ time.
In our case $|V| \leq m+n$, $|E| \leq  m^2$.
Thus, it can be checked whether $x$ is avoidable in $O(m^2\sqrt{m+n} )$ time.
Consequently, listing of all unavoidable elements by checking all the elements of $B$ can be done in $O(m^2n\sqrt{m+n})$ time. The bounds without $n$ follow from the fact that we can limit the number of relevant elements in $B$ easily to $m^2$ (the number of elements in the first $m$ columns of the matrix).
\end{proof}

Recall a set $X$ is said to be \emph{avoidable} if there exists a matching that avoids every element of $X$. Note that it is possible that all elements of a set are avoidable, while the set itself is not.
Theorem~\ref{thm:unavoidable} extends to set $X$, by replacing all elements of $X$ by a single element $x$. If $x$ is avoidable, then $X$ was avoidable and vice versa.
(If more than one element of $X$ appears in a row replace only the leftmost one.)




\section{Complexity of reachability}\label{sec:Compl}
In this section we show that Problems 1 to 5 defined in the introduction lie in the indicated complexity classes.

Problem~\ref{prb:DER} is polynomially solvable, seen as follows: let $M$ be a matrix and $E$ be a set of $m$ elements occurring in the matrix. We define $F:= B \setminus E$, then $E$ is exactly reachable if and only if $F$ is avoidable, which can be checked in polynomial time by Corollary~\ref{thm:unavoidablecompl}.



We start with a brief and informal explanation of the class \#P. A detailed introduction is given by Arora and Barak~\cite{arora2009computational}.

Informally \#P is a class of counting problems. In counting problems
an object is given and the task is to count some other objects satisfying some condition. The only restriction is that the size of the binary representation of the number of objects counted is polynomial in the input. But this is given in all that follows, because we have exponential upper bounds on the number of objects to be counted.

Famous examples of \#P-complete problems are:
\begin{itemize}
   \item How many variable assignments does a given boolean formula satisfy?
   \item How many perfect matchings does a given bipartite graph have?
\end{itemize}

Note that the decision version of the first problem is NP-hard and of the second problem is in P.

A problem $L$ is defined to be \#P-complete if it lies in \#P and there exists a polynomial time counting reduction from all other problems in \#P to $L$. Due to transitivity of counting reductions, it is sufficient to  reduce one \#P-complete problem to $L$.

Finally, a counting reduction transforms the input of one problem to the input of a second problem in a way that the counting problem for the first problem can be solved via the reduction and the solution to the second problem. It will be very clear in our context that our reductions are indeed counting reductions.

We will use the result from Creignou and Hermann~\cite{creignou1993p} that \#$1$-in-$3$-SAT is \#P-complete. Here it is asked how many assignments exist such that every clause of a given $3$CNF-formula has exactly one true and two false literals. We will also use later that the corresponding decision problem is NP-complete.


We start our discussion with Problems~\ref{prb:DR}~and~\ref{prb:CERSS} for $3$-column matrices.
\begin{theorem}\label{thm:ReductionSAT}
    \ Problem~\ref{prb:DR} (Deciding Reachability) is NP-complete and Problem~\ref{prb:CERSS} (Counting Exactly Reachable Supersets)  is \#P-complete, even when $D$ (the set we want to decide if it is reachable) is a $1$-element set and the matrix has only $3$ columns.
\end{theorem}

 We will give a transformation that converts $1$-in-$3$-SAT formulas $\Phi$ to matrices $M$ with the special property that the number of good assignments for $\Phi$ equals the number of reachable sets containing some special element $x$.
This transformation will reduce $1$-in-$3$-SAT to Problem~\ref{prb:DR} and \#$1$-in-$3$-SAT to Problem~\ref{prb:CERSS} and thus proves the essential part of Theorem~\ref{thm:ReductionSAT}.
\begin{figure}[htp]
   \[
     M_{\Phi} = \left(\begin{array}{cccccc}
      a_1 & x_1 &  * & *& &\\
      a_1 & \overline{x_1} & * & *& &\\
       \vdots & \vdots & & & &\\
      a_n & x_n &  *& *& &\\
      a_n & \overline{x_n} & *& *\\
       & & & & \\
      b_1 & c_1 & L_1^1& \overline{L_2^1}&  \overline{L_3^1} & C_1\\
      b_1 & c_1 & L_2^1& \overline{L_1^1}&  \overline{L_3^1}& C_1\\
      b_1 & c_1 & L_3^1& \overline{L_1^1}&  \overline{L_2^1}& C_1\\
       & & & & \\
       & & \vdots & \vdots   & \\
       & & & &\\
       b_m & c_m & L_1^m& \overline{L_2^m}&  \overline{L_3^m} & C_m\\
      b_m & c_m & L_2^m& \overline{L_1^m}&  \overline{L_3^m}& C_m\\
      b_m & c_m & L_3^m& \overline{L_1^m}&  \overline{L_2^m}& C_m\\
      & & & &\\
      C_1 & C_2 & \ldots & & C_m & x
    \end{array}\right)
   \] \caption{Illustration of the matrix $M_{\Phi}$}
   \label{fig:Matrix}
   \end{figure}

\begin{lemma}\label{lem:reduction}
  Let $\Phi$ be an instance of $1$-in-$3$-SAT. Then there exists a  matrix $M$ with some element $x$ such that the $1$-in-$3$-satisfying truth assignments of the variables of $\Phi$ are in one-to-one correspondence with the exactly reachable sets $E$ with $x\in E$.
\end{lemma}
\begin{proof}
   Let $\Phi$ be of the form $ C_1 \wedge C_2 \wedge \ldots \wedge C_m$. And each of the clauses $C_i$ is of the form $L_1^i\vee L_2^i \vee L_3^i$, where each literal $L_j^i$ is one of the variables $x_1,\ldots ,x_n$ or its negation.
   The matrix $M_{\Phi}$ is defined in Figure~\ref{fig:Matrix}.

   Here $a_i, b_i , c_i$ are only used where indicated. The asterisks mark that any element could stand there. Every $x_i$ and its negation get an element named after them. The same holds for every clause $C_i$. The literals $L_i^j$ need to be replaced by the corresponding variable element.

   Let $F : \{x_1,\ldots , x_n\}\longrightarrow \{0,1\}$ be an assignment of the variables such that exactly one literal of each clause in $\Phi$ becomes true. We construct the corresponding matching $\tau_F$. The matching $\tau_F$ selects $x_i$  whenever $F$ gives it the value $1$ and its negation otherwise.

There exists exactly one way to select the element $C_i$ for each $i$.
To see this consider as an example $C = x_1 \vee x_2 \vee x_3$ and assume w.l.o.g. that $x_1$ is true and $x_2$ and $x_3$ are false in assignment $F$. Then $\tau_F$ has already selected $x_1,\overline{x_2}, \overline{x_3}$. Thus $C$ can be selected in the following way, indicated by circles:

\begin{center}
\includegraphics{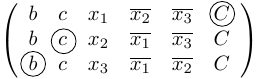}	
\end{center}

In the same way all clause labels can be selected and it is possible to select~$x$.

   In this manner we get from every satisfying assignment $F$ a matching $\tau_F$ selecting $x$. It is easy to see that $\tau_F$ is a $1$-POM. And thus there also exists a POM selecting the same set.
   In this way we get from every satisfying truth assignment a $1$-POMs selecting $x$.

   We show that we can get a satisfying truth assignment for every $1$-POM $\tau$ selecting $x$.
We get a truth assignment by taking $x_i$ to be true if $x_i$ is selected in the first $2n$ rows and $x_i$ false if $\overline{x_i}$ is selected. $\tau$ will always select one of $x_i$ and $\overline{x_i}$.
Clearly, $C_1, \ldots,C_m$ are not taken by the last row.  The only way to select $C_i$ is if one of its three literals is assigned to be true and the other two are assigned to be false.
Thus this assignment is an $1$-in-$3$-satisfying assignment.
\end{proof}

   Let $M$ be some matrix with the element $x$. We define a matrix $M'$ with exactly three columns by replacing each row by a block. Let $r = (a_1 a_2 \ldots a_m)$ be some row, with $m>3$. We replace it by the following block $b_r$:
   \[ \left(
   \begin{array}{ccc}
	a_1 & a_2 & \beta_1 \\
	\beta_1 & a_3 & \beta_2 \\
	\beta_2 & a_4 & \beta_3 \\
	 & \vdots & \\
	 \beta_{m-4} & a_{m-2} & \beta_{m-3} \\
	 \beta_{m-3} & a_{m-1} & a_m  \\
	
   \end{array}
\right)\]
The $\beta_i$ elements are different for every block.
If $\tau$ is a Pareto optimal matching of $M$, then there exists a unique Pareto optimal matching $\tau'$ in $M'$ that selects in each block the same element as $\tau$ in $M$.
To see this correspondence assume element $a_i$ is selected in our row $r$.
Then the elements $a_1,\ldots,a_{i-1}$ are selected somewhere else in $M$.
Associate with this matching of $M$ the matching in $M'$ that in block $b_r$ selects the elements $\beta_1,\ldots,\beta_{i-2}, a_i , \beta_{i-1},\ldots,\beta_{m-3}$ consecutively from the top row to the bottom row.
%
In the special case
that $i = 1$, $i=2$, $i=m-1$ or $i=m$ either the first or the second block of
$\beta$s will be empty.
It is easy to check that this is a Pareto optimal matching of $M'$ and has the above property.
Uniqueness follows from the fact that in each block $b_r$
exactly one of the original $a_i$ must be selected.

In every block at most one element of the original row can be selected by any POM $\tau'$.
Every POM $\tau'$ in $M'$ that selects in every block exactly one element of the original row corresponds to exactly one POM $\tau$ of $M$.
Consider the three column matrix $M'$ of the matrix $M$ in Lemma~\ref{lem:reduction}. Clearly, the POMs selecting $x$ in $M$ stand in a one to one correspondence with the POMs in $M'$ selecting $x$.
The next lemma summarizes our discussion.


\begin{lemma}
  Let $\Phi$ be some boolean $3$-DNF formula. Then there exists a \textbf{three column} matrix $M$ with some element $x$ such that the $1$-in-$3$-satisfying truth assignments of the variables of $\Phi$ are in one-to-one correspondence with the exactly reachable sets $E$ with $x\in E$.
\end{lemma}

\begin{proof}[Proof Theorem~\ref{thm:ReductionSAT}]
\#$1$-in-$3$SAT is \#P-complete. Given a $3$-DNF formula $\Phi$, the reduction above gives a $3$-column matrix, such that the exactly reachable sets containing $x$ are in a one to one correspondence with the satisfying assignments. Thus also Counting Exactly Reachable Supersets is \#P-hard. And as $1$-in-$3$SAT is also NP-complete, deciding Reachability is also NP-hard. Membership to these classes (and thus completeness) follows from the exponential upper ($2^{m^2}$) bound on the number of possible sets. Thus the binary representation is linear in the size of the input.
\end{proof}
The natural question to ask is what happens when there are only \textbf{two columns}. Understanding the structure of reachability for $2$ column matrices gives many interesting results at once.
We consider first Problem~\ref{prb:DR} (Deciding Reachability) and will see easily from the algorithm that we can find an explicit way to count all reachable sets, for $2$ column matrices.
The complexity of counting reachable sets for general matrices remains open.
On the other hand, we will see that it is \#P-complete to count all exactly reachable sets $E$ already for $2$-column matrices.



%


\begin{theorem}\label{thm:polyreachable}
  Problem~\ref{prb:DR} (Deciding Reachability) is in P for $2$-column matrices.
\end{theorem}
\begin{proof}
We assume the bipartite row element graph $G$ is connected, otherwise we treat each component separately.
Let $m$ denote the number of rows, $D$ the set of elements we wish to reach, and  $X$ the set of elements of $G$.

   It is easy to see, by induction on $m$, that the number of different elements in $G$ is at most $m+1$. (This is true for $m= 1$ and at most one new element is added when a new row is added.)

   We distinguish two cases.
   \begin{enumerate}[label = Case (\arabic*), leftmargin=*]
   		\item  $|X| = m+1$
      \item  $|X| \leq m$
   \end{enumerate}

   In the first case we cannot reach $X$ but we will show that for every avoidable element $x\in X$ we can reach $X\setminus \{x\}$. This implies $D$ is reachable if and only if it does not contain all avoidable elements of this component.

   In the second case we will show that $X$ (and subsequently $D$) can be selected.
	
In the first case, there is one more element than the number of rows and thus it is clear that at least one avoidable element cannot be selected.
Further, the number of vertices is $2m+1$ and the number of edges is $2m$. This implies that $G$ is a tree. Let $x$ be any avoidable element not in $D$. We orient all edges away from $x$ in $G$. Clearly the in-degree of every element is exactly one. For every edge oriented from row $r$ to element $e$, let $r$ select $e$. As $x$ is avoidable, it is not in the first column. As all elements, except $x$ is selected there is no $1$-blocking coalition.
Thus the described matching is a $1$-POM, by Lemma~\ref{lem:equivNotion} there is a POM selecting $D\subseteq X\setminus \{x\}$.

Consider now case 2.
As every row is incident to exactly two elements the number of edges is exactly $2m$, and as the number of vertices is at most $2m$, $G$ contains a cycle.
Let $r_0,e_0,r_1,e_1, \ldots , r_t,e_t$ be our cycle.
We assign $r_i$ to $e_i$. In this way all elements on the cycle are selected.

We will repeat in the following. Pick a row $r$ that has at least one of the elements already selected. Such a row exists because we assume that $G$ is connected. There is at most one element which can be selected by $r$. Let $r$ select this element if possible.

This procedure leads to a matching $\tau$ that selects all elements and thus $\tau$ is automatically $1$-POM, by Lemma~\ref{lem:equivNotion} there is a POM selecting $D\subseteq X$.


All the above steps are computationally easy.
\end{proof}

    Let $M$ be a $2$ column matrix.   We assume the corresponding bipartite row element graph has $k$ components. Denote by  $a_i$ the number of avoidable elements $E_i$ of the $i$th component and by $u_i$ the number of unavoidable elements $F_i$. Recall that it is easy to check if an element is avoidable
    by Corollary~\ref{thm:unavoidablecompl}.
    Further $\chi_i$ is the indicator variable for the event that the $i$th component is a tree. Recall that every subset of a reachable set is reachable.
    \begin{theorem}\label{lem:numberSelectable}
       With the notation and assumptions above:
       \[ |\{ E: E \textrm{ is reachable }\}|  =
        \prod_{i=1}^{k} \left(2^{a_i}-\chi_i\right)2^{u_i}\]
        This implies Problem~\ref{prb:CR} (Counting Reachable Sets) is in P for $2$ column matrices.
    \end{theorem}
    \begin{proof}
      It is clear that we can look at each component of the bipartite row element graph separately.
      Consider the case that component $i$ is a tree.
      We know from the proof of Theorem~\ref{thm:polyreachable} that every proper subset of $E_i$ is reachable, i.e., at least one element must not be chosen. These are $2^{a_i} -1$ subsets. From the unavoidable elements all subsets can be chosen.
      Consider now the case that component $i$ has a cycle. Then we know that all elements from this component can be chosen at once. This leads to $2^{a_i+u_i}$ reachable subsets.
      As the choice is independent for each component, we take the product.
    \end{proof}
    \begin{theorem}\label{thm:problem5}
      \  Problem~\ref{prb:CER} (Counting Exactly Reachable Sets) is \#P-complete already for $2$-column matrices.
    \end{theorem}
    We will show \#P-hardness for $2$-column matrices. By adding $k-2$ columns each consisting of a single new element we get a $k$-column matrix with the same number of exactly reachable sets. Thus the problem is also \#P-hard for $k$-column matrices with $k\geq 2$.

\begin{proof}
    We reduce the problem from the \#Independent Set Problem, which is \#P-complete. This is the problem of counting the number of independent sets in a given graph. Membership in \#P follows from the exponential bound ($2^{m^2}$) on the number of exactly reachable sets.
    More on the complexity of counting independent sets in a graph can be found in~\cite{vadhan1995complexity,greenhill2000complexity}.

    Let $G=(V,E)$ be some connected graph. We construct a two column matrix $M_G$ with the property that non-empty independent sets are in one-to-one correspondence with exactly reachable sets.
        The elements of $M_G$ are the edges and the vertices of $G$. For each edge $e = (u,v)$ we insert two rows $(e,u)$ and $(e,v)$.
This implies, that the edges are unavoidable and for every edge $e$ at most one of its vertices will be selected in the two rows corresponding to $e$.

    Let $W\subseteq V$ be some non-empty independent set. The set $X_W$ is defined as $X_W = E \cup (V \setminus W)$. We show: $X_W$ is exactly reachable.
    We do this in three small steps. First we define an orientation on $G$. Then we define a permutation $\pi$ on $M_G$. At last we will argue that the greedy matching $\tau_\pi$ defined by $\pi$ selects exactly $X_W$.
    Let $T$ be some spanning tree of $G$ and let $x$ be any element in $W$. We orient every edge away from $x$ in $T$, to get an orientation $O_1$ on $T$. Note except for $x$ every vertex has indegree $1$ in $T$.

    We construct $O_2$ from $O_1$ by changing the orientation of each edge incident to $W$ such that the in-degree of every vertex in $W$ is zero and orient every remaining edge in $G$ arbitrarily. For this orientation holds $v\in W$ if and only if $indegree(v) = 0$.

Let $\pi$ be any permutation such that the following holds:
for every edge $e = (u,v)$ oriented towards $v$, row $(e,u)$ is before $(e,v)$ in $\pi$.
This implies that the vertices not selected by $\tau_\pi$ are exactly those with in-degree zero. This shows $X_W$ is exactly reachable.


    Conversely, let $X$ be some exactly reachable set. All edge elements are unavoidable, thus are in the set $X$. Further the vertex elements $W$ not selected form an independent set in the graph. To see this observe that two adjacent vertex elements cannot be avoided simultaneously.
    Note that $W$ can be empty.

    We have shown that every non-empty independent set $W$ in $G$ corresponds to an exactly reachable set $X$ in $M_G$.
    Note that the empty set is also independent in $G$ and would correspond to $X = E\cup V$.
    Note $E\cup V$ is not necessarily exactly reachable, but we can easily check if this is the case.
    (We leave it as an exercise to check that this is the case whenever $G$ is a tree.)

    Thus the number of exactly reachable sets equals the number of independent sets, maybe minus $1$. And we can determine efficiently if this ``minus $1$'' is the case.
   \end{proof}

This last theorem also implies \#P-hardness of Problem~\ref{prb:CERSS} (Counting Exactly Reachable Supersets) for $2$-column matrices with $D=\varnothing$.

\section{Multi-matchings and miscellany}\label{sec:esmm}

Our notions about matchings can be generalized naturally to the case of {\em multi-matchings}, defined in the following way. Additional to $A$ and $B$ and the preference lists of members of $A$ (i.e., $M$), we are given a list of prescribed degrees $L \subset \N $. For each $a\in A$, there exists a degree $\ell_a \in L$ associated with $a$. We are interested in subgraphs of the complete bipartite graph between $A$ and $B$ where the degree of every $a$ is exactly $\ell_a$. Clearly, $\ell_i=1$ for all $i$ gives the case of matchings.
This can be pictured such that each member of $a$ has to be matched with $\ell_a$ members of $B$.
In the matrix terminology it means that in row $a$ we select $\ell_a$ positions. Note that for the case of multi-matchings a matrix may have more than $m$ relevant columns. In a given row $r$ we say that a set of positions $P=\{p_1,p_2,\dots ,p_{\ell_r}\},p_1<p_2<\dots <p_{\ell_r}$ (a position is smaller than another if it is to the left of the other) is {\em better} than another, different set $Q=\{q_1,q_2,\dots ,q_{\ell_r}\},q_1<q_2<\dots <q_{\ell_r}$ (both of size $\ell_r$) if $p_i\le q_i$ for every $i$ (and $p_i< q_i$ for at least one $i$).
(Note that  two sets of positions might be incomparable.)

A \emph{blocking coalition} of a multi-matching $\tau$ is a subset $A'$ of $A$ such that there exists a multi-matching $\tau'$ that differs from $\tau$ on all elements of $A'$ and the selected set of positions is better in $\tau'$ than in $\tau$ on $A'$.
If there exists no blocking coalition (resp.\ of size $i$), we call the matching $\tau$ an \emph{Pareto optimal multi-matching} (POMM) (resp.\ $i$-POMM).

A multi-matching $\tau$ is \emph{greedy} if there exists a multiset permutation $\pi$ of the multiset of rows in which each row $r\in A$ appears $\ell_{r}$ times, for all $r$, such that the multi-matching can be generated in the following manner: we process the rows of $M$ in the order determined by $\pi$, and in each step we pick the leftmost element in the current row that was not picked earlier.


First we present a general method how to generalize statements about POMs to statements about POMMs. Let $(M, L)$ be a pair where $M$ is a matrix with $m$ rows and $L=(\ell_1, \dots, \ell_m)$ is the list of degrees.
Let $M'$ be the matrix with $\sum_{1 \leq r \leq m} \ell_r$ rows
in which the $r$th row of $M$ appears $\ell_r$ times.
We can label the rows of $M'$ by pairs $(\alpha, \beta)$
where $1 \leq \alpha \leq m$,
$1 \leq \beta \leq \ell_\alpha$
(denote the set of all such pairs by $S$).
The row of $M'$ labeled by $(\alpha, \beta)$ is the $\beta$th occurrence of the $\alpha$th row of $M$.

\begin{lemma}\label{multi}
The sets and the elements that can be chosen by greedy matchings from $M'$
are the same as
the sets and the elements that can be chosen by greedy multi-matchings with degree list $L$ from $M$.
\end{lemma}

\begin{proof}
In a multiset permutation $\pi$ of $[m]$, where $r$ can be chosen $\ell_r$ times,
replace the $\beta$th occurrence of $\alpha$ by the pair $(\alpha, \beta)$.
In this way $\tilde{\pi}$, a permutation of $S$, is obtained.
The correspondence $\pi \leftrightarrow \tilde\pi$ is clearly a bijection (up to equivalence on permuting in $\tilde\pi$ the occurrences of copies of the same row).

It is clear from the definitions that the set chosen by $\pi$ from $M$
is equal to the set chosen by $\tilde\pi$ from $M'$.
\end{proof}

Using this observation, our results about POMs can be easily generalized to POMMs.
 Recall the differences between POMs and POMMs.
 Additional to $A$,$B$ and $M$, we are given a set of integers
 $L=(\ell_1,\dots \ell_m)$.
 In the definition of greedy-POMMs instead of
 permutations we regard multiset permutations that are sequences of rows which have length $\ell=\ell_1+\ell_2+\dots+\ell_m$ and contain row $i$ as often as $\ell_i$.
 The $\ell_i$ appearances of a given row are not necessarily consecutive in the sequence.
 Observation \ref{obs:triv} and Lemma \ref{lem:equivNotion} generalize easily to this
 setting. Also, Theorem \ref{thm:elements} can be
 generalized for this case. We summarize the generalization of the latter in the
 following theorem.

\begin{theorem}\label{thm:ESMM}
Given $A$,$B$,$M$ and $L$ as above. Let $T$ be a set of $k$ POMMs and $E(T)$ the set of elements selected by at least one multi-matching of $T$, then \[\left| E(T) \right| \leq \sum_{i=1}^{\min(k,m)}{\left\lfloor \frac{\ell}{i}\right\rfloor}.\]
\end{theorem}

We leave the proofs for the interested reader as they follow closely the corresponding proofs for POMs.
For $k=m$ this theorem is again asymptotically tight (Example \ref{ex:mlogmExample} remains trivially good, yet note that the asymptotics depends on $L$) and for fixed $k$ and big $m$ it is tight if $\ell_1=\ell_2=\dots=\ell_m$.
Indeed, in this case
the construction from Claim~\ref{cl:mlgmtight} below
can be inflated such that for every element $a$ in the example we introduce $\ell_1$ new elements $a^1,a^2,\dots a^{\ell_1}$ and in the example matrices we change every element $a$ to the ordered list of elements $a^1,a^2,\dots a^{\ell_1}$. This way the number of columns is multiplied by $\ell_1$. In the examples the relevant permutations are also inflated similarly, when in a step we would choose row $r$ now instead we choose the same row $r$ as often as  $\ell_1$. Note that in this case (when all $\ell_1$s are the same) this trick can be used in Example \ref{ex:mlogmExample} as well, and the asymptotics become independent of $L$.


Theorem~\ref{thm:unavoidable} can be generalized for multi-matchings as follows.

\begin{theorem}\label{lem:avoidable}
Let $(M, L)$ be as before ($M$ is a matrix with $m$ rows, $L=(\ell_1, \dots, \ell_m)$ is the list of degrees).
An element $x$ of $M$ is avoidable
if and only if
for every set $R$ of rows of $M$, we have:
	\[|E(R)| \geq  \sum_{r\in R} \ell_r .\]
\end{theorem}
\begin{proof}
This follows from Observation~\ref{multi}
and from the fact that the condition from this theorem for $(M, L)$
is clearly equivalent to the condition from Theorem~\ref{thm:unavoidable} for $M'$.
\end{proof}
Finally (back to POMs) we mention some miscellaneous results that give some insight into the structure of matrices with the biggest number of different exactly reachable sets.

\begin{lemma}
	Let $M$ be a matrix and $e_1,e_2,\ldots,e_k$ unavoidable elements. Further let $M'$ be the matrix obtained by deleting $e_1,e_2,\ldots,e_k$ in $M$ and appending $k$ new columns to the left of $M$. The $i$-th column  consists of the element $e_i$ in every entry.
	Let $E$ be an exactly reachable set of $M$, then it is an exactly reachable set also in $M'$. In particular $|\mathcal{E}(M')|\ge |\mathcal{E}(M)|$.
\end{lemma}
\begin{proof}
	Let $E$ be some exactly reachable set and $\pi$ some permutation for which $s(\tau_\pi)=E$. We modify $\pi$ to $\sigma$ by moving the rows that selected $e_1,e_2,\ldots,e_k$ to the beginning of the permutation. We claim that $\sigma$ selects $E$ in $M'$. Obviously, $e_1,\ldots,e_k$ are selected in the first $k$ steps. Afterwards, $\sigma$ behaves on $M'$ exactly as $\pi$ on $M$.
\end{proof}

This means that we can transform every matrix, such that all the unavoidable elements are in the first columns, without reducing the number of exactly reachable sets. We show another transformation that does not reduce the number of exactly reachable sets.

\begin{lemma}
	Let $M$ be a matrix, then there exists a matrix $M'$ such that for any row the element in the last reachable position of that row does not appear anywhere else in the matrix and $|\mathcal{E}(M')|\ge |\mathcal{E}(M)|$.
\end{lemma}
\begin{proof}
	We achieve one by one for each row that the element in the last reachable position in that row does not appear anywhere else in the matrix. Say, in a row $r$ the element $e$ in the last reachable position $p$ is not unique in $M$. We get $M'$ by simply changing $e$ to some new element $e'$ in position $p$. Now for some permutation $\pi$, $s(\tau_\pi)$ either remains the same in $M'$ and $M$ or not.
	
	In the latter case $\tau$ in $M$ must choose position $p$. We claim that it cannot happen that after choosing $p$, $\tau$ chooses a position in some other row which has element $e$ to its left. Indeed, in this case take the first such row $r'$ after $r$ in $\pi$. If we exchange $r$ and $r'$ in $\pi$, we get a permutation $\pi'$ such that the corresponding matching $\tau'$ chooses in $M$ the element next to the right to $p$. This is a contradiction as $p$ was a last reachable position. This implies that $\tau$ chooses exactly the same positions in $M$ and $M'$ and thus also the same elements, except that in $M'$ it chooses $e'$ instead of $e$.
	
	Summarizing, for every image set $I$ in $\mathcal{E}(M)$ either $I$ is also in $\mathcal{E}(M')$ or $I'$ is in $\mathcal{E}(M')$, where we get $I'$ by changing $e$ (which must be in $I$ in this case) to $e'$ in $I$. As $e'$ is a new element, for
	different image sets $I_1,I_2$, this way we get different
	image sets $I'_1,I'_2$, thus $|\mathcal{E}(M')|\ge |\mathcal{E}(M)|$.
\end{proof}


\vfill

\paragraph{Acknowledgments}
We want to thank Matthias Henze and Rafel Jaume
for posing Theorem~\ref{thm:elements} as an open problem.
We also want to thank Rob Irving, \'{A}gnes Cseh and David Manlove
for helping us to find related work to our problem.
Special thanks goes to Nieke Aerts
for enjoyable and interesting discussions
on attempts to improve Corollary~\ref{crl:mlogm_choose_m}.

We would also like to thank the annonymous reviewer
for helpful comments on the presentation and on
links to the literature.

We thank Nick Brettell for helpful comments on English grammar.

Research of Andrei Asinowski was supported by the
ESF EUROCORES programme EuroGIGA, CRP `ComPoSe',
Deutsche Forschungsgemeinschaft (DFG), grant FE 340/9-1.

Bal\'azs Keszegh is supported by the  Hungarian National Science Fund (OTKA),
under grant PD 108406 and by the J\'anos Bolyai Research Scholarship of the Hungarian Academy of Sciences.
This research was partially done while Bal\'{a}zs Keszegh was at FU Berlin
in the scope of an EuroGIGA Cross-CRP visit and later with a DAAD Study Visit Grant for Senior Academics.

The research of Tillmann Miltzow is partially supported by the ERC grant PARAMTIGHT: "Parameterized complexity and the search for tight complexity results", no. 280152.

\newpage

\bibliographystyle{plain}

\bibliography{CountingESMBib}
\end{document}